\documentclass[preprint,12pt]{elsarticle}
\usepackage[T1]{fontenc}

\usepackage{graphicx}
\usepackage{thm-restate}
\usepackage{amsmath}
\usepackage{verbatim}
\usepackage{amssymb}
\usepackage{amsthm}
\usepackage{multirow}

\theoremstyle{definition}
\newtheorem{definition}{Definition}[section]

\theoremstyle{plain}
\newtheorem{prop}{Property}[section]
\newtheorem{theorem}{Theorem}
\newtheorem{conj}{Conjecture}

\newtheorem{corollary}{Corollary}

\usepackage{hyperref}

\newcommand{\cf}[2]{$\mathsf{Conflict}(#1,#2)$}

\usepackage{tikz}
\usetikzlibrary{arrows, quotes}
\usetikzlibrary{arrows.meta}

\begin{document}

\begin{frontmatter}

\title{Antimagic labelling of graphs with maximum degree $\Delta(G) = n-4$}
\author{Grégoire Beaudoire, Cédric Bentz, Christophe Picouleau}

\affiliation{organization={CEDRIC, Conservatoire National des Arts et Métiers},city={Paris}, country={France}}

\begin{abstract}

An antimagic labelling of a graph $G = (V,E)$ is a bijection from $E$ to $\{1,2, \ldots, |E|\}$, such that all vertex-sums are pairwise distinct, where the vertex-sum of each vertex is the sum of labels over edges incident to this vertex. A graph is said to be antimagic if it has an antimagic labelling. It has been proven that graphs $G$ with $\Delta(G) \geq n - 3$ are antimagic, where $\Delta(G)$ is the maximum degree of a vertex in $G$ and $n = |V|$. In this article, we extend this result to graphs with $\Delta(G) = n - 4$, provided that $|E| \geq 7n$.

\end{abstract}


\end{frontmatter}

\section{Introduction and definitions}

In this paper, we only consider finite, simple and undirected graphs. Let $G = (V,E)$ be a graph, with $|V| = n$ and $|E| = m$. We denote by $d_G(v)$ the degree of a vertex $v \in V$. If the graph $G$ is clear from the context, we will simply write $d(v)$. Let $\Delta(G) = \max\limits_{v \in V} d(v)$.

Two vertices $u,v \in V$ are called \emph{neighbours} if $uv \in E$. We will denote by $\Gamma(u)$ the (open) neighbourhood of $u$, i.e. $\Gamma(u) = \{v | uv \in E\}$.

Given a graph $G = (V,E)$, let $f : E \rightarrow \{1,2,\ldots, m\}$ be a bijective labelling of the edges of $G$. For each vertex $u \in V$, we will write $\sigma(u) = \sum\limits_{v \in V | uv \in E} f(uv)$ the sum of labels over edges incident to $u$. If all the values of $\sigma(u)$ are pairwise distinct, then $f$ is called an \emph{antimagic labelling} of $G$. If $G$ admits at least one antimagic labelling, $G$ is said to be \emph{antimagic}.

Antimagic labelling was originally introduced by Hartsfield and Ringel in 1990 \cite{hartsfield1990}, with the following conjecture:

\begin{conj}
    Every connected graph other than $K_2$ is antimagic.
\end{conj}

The topic is the focus of a chapter of 12 pages in the dynamic survey, updated yearly, on graph labelling by J. Gallian \cite{survey}.

Our paper focuses on graphs with large maximum degree. A first result, easily provable, is the following:

\begin{theorem}\cite{alon2004}
    \label{n-1}
    Graphs $G$ with $\Delta(G) = n - 1$ are antimagic.
\end{theorem}

We give the following proof of this theorem, since the core idea described here will be useful to understand our construction later on.

\begin{proof}
    Let $G = (V,E)$ be a graph with a vertex $r$ such that $d(r) = \Delta(G) = n - 1$, meaning that $r$ is a neighbour of every other vertex in $G$. Arbitrarily label every edge $uv \in E$ with $u,v \neq r$, using all the labels in $\{1,2,\ldots,m - (n-1)\}$. Every vertex (different from $r$) has, at this point, all of their incident edges labelled except one. Sort the vertices of $G$ (different from $r$) by increasing value of their partial sum $\sigma'(v_1) \leq \sigma'(v_2) \leq \ldots \leq \sigma'(v_{n-1})$ ($\sigma'(x)$ will denote the partial sum over labelled edges incident to the vertex $x$). We then label the edge $rv_i$ with $m - (n-1) + i$. We obtain $\sigma(v_1) < \sigma(v_2) < \ldots < \sigma(v_{n-1})$. Moreover, $\sigma(v_{n-1}) < \sigma(r)$ since $r$ is a vertex with maximum degree in $G$, and all the edges incident to $r$ are labelled with the largest labels. Overall the labelling is antimagic.
\end{proof}

Alon \emph{et al.}, in 2004, proved the following result \cite{alon2004}:

\begin{theorem}
    Graphs $G$ with $\Delta(G) = n - 2$ and $n \geq 4$ are antimagic.
\end{theorem}

The next result over graphs with large maximum degree comes from Yilma, in 2013 \cite{yilma2013}:

\begin{theorem}
    \label{yilma}
    Connected graphs $G$ with $\Delta(G) = n - 3$ and $n \geq 9$ are antimagic.
\end{theorem}

In the same paper, they also proved the following theorem:

\begin{theorem}
    \label{yilmak2}
    Let $G = (V,E)$ be a graph, and $x$ a vertex such that $d(x) = \Delta(G) = n - k$, with $k \leq \frac{n}{3}$. Let $y$ be a neighbour of $x$ such that $\Gamma(x) \cup \Gamma(y) = V$. Then $G$ is antimagic.
\end{theorem}

Since 2013, to our knowledge, there has been no advances over antimagic property for graphs with large maximum degree. There have been results over graphs with specific properties (for instance, regular graphs are antimagic \cite{Berczi2015} \cite{chang2015}) but the following conjecture, introduced by Alon \emph{et al.} in 2004, is still open:

\begin{conj}
    \label{conjalon}
    Let $G = (V,E)$ be a connected graph, with a vertex $x$ such that $d(x) = \Delta(G) = n - k$ for some $k \geq 4$. Then, if $n \geq n_0(k)$ for some $n_0(k)$, $G$ is antimagic.
\end{conj}

In this paper, we extend the results over graphs with a large maximum degree by proving the following theorem:

\begin{theorem}
    \label{n-4}
    Let $G = (V,E)$ be a connected graph, with $\Delta(G) = n - 4$ and $m \geq 7n$. Then $G$ is antimagic.
\end{theorem}

The rest of the article is mostly devoted to the proof of Theorem \ref{n-4}. In Section $2$, we explain the basic algorithm to label the graph, and we show some properties of this labelling. In Section $3$, we explain how to exchange some labels in the general case in order to obtain an antimagic labelling. In Section $4$, we show how to deal with specific cases not included in Section $2$ and $3$. Those three sections make up together the proof of Theorem \ref{n-4}. In Section 5, we extend Theorem \ref{n-4} to non-connected graphs. Finally, we give some open questions related to this work in the conclusion.

\section{Labelling process}

Let $G = (V,E)$ be a connected graph with a vertex $r$ such that $d(r) = \Delta(G) = n - 4$. We will call $u_1,u_2,u_3$ the three vertices of $V$ that are not neighbours of $r$, and $H$ the subgraph induced by $V \setminus \{r,u_1,u_2,u_3\}$. We will denote by $n_H$(respectively $m_H$) the number of vertices (respectively the number of edges) of $H$. We will denote by $d'(u_1)$ (respectively $d'(u_2),d'(u_3)$) the number of edges $u_1v$ (respectively $u_2v$, $u_3v$), with $v \in H$. We assume that $d(u_1) \geq d(u_2) \geq d(u_3)$, and that if $d(u_i) = d(u_{i+1})$ for some $i$, then $d'(u_i) \geq d'(u_{i+1})$.

Notice that this order implies that $d'(u_1) \geq d'(u_2) \geq d'(u_3)$. Indeed, if there exists $i$ such that $d'(u_{i+1}) > d'(u_i)$, since $d(u_{i+1}) \leq d(u_i)$ and since the difference between the degree of $u_i$ and the degree of $u_{i+1}$ in the subgraph induced by $\{u_1,u_2,u_3\}$ is at most $1$, we obtain $d(u_{i+1}) = d(u_i)$. However, this implies $d'(u_{i+1}) \leq d'(u_i)$, and we obtain a contradiction.

Notice that $m \geq 7n$ induces a lower bound for $n$. Indeed, we have $m_H \geq 7n - 4(n-4)$ since the vertices $r,u_1,u_2,u_3$ all have degree at most $n - 4$. We also have $m_H \leq \frac{1}{2}(n-4)(n-5)$ since $H$ is made of $n-4$ vertices with degree at most $n-5$ (since $\Delta(G) = n - 4$ and all the vertices of $H$ are neighbours of $r$). By combining the two inequalities, we obtain:

\begin{align*}7n-4(n-4) &\leq \frac{1}{2}(n-4)(n-5) \\
    6n + 32 &\leq n^2 -9n + 20 \\
    0 & \leq n^2 -15n - 12
\end{align*}

Hence $n \geq 16$.

Note that, if there exists a vertex $v \in H$ such that $v$ is a neighbour of all three $u_1,u_2,u_3$, then Theorem \ref{yilmak2} applies if $n \geq 12$ (which is the case here, as we just showed), and $G$ is antimagic. We will then assume in the following that there is no such vertex.

Throughout Sections 2 and 3, we assume that $d'(u_1) \geq d'(u_2) \geq d'(u_3) \geq 4$, and that the graph induced by $\{u_1,u_2,u_3\}$ is an independent set. We show in Section 4 how to obtain an antimagic labelling if those conditions are not met. At any point during the labelling of a graph, for each vertex $u$, we will denote by $\sigma'(u)$ the partial sum over labelled edges incident to $u$. Once the labelling is done, all edges are labelled and $\sigma'(u) = \sigma(u)$.

We now describe the first stage of the labelling process. We partition $E = E_1 \cup E_2$ such that $E_1$ is the set of edges incident to $r$. We will first label $E_2$, then $E_1$, by reserving the following labels:

\begin{itemize}
    \item Edges in $E_1$ will be labelled with $\{m, m-4, m-8, \ldots, m-4(n-5)\}$.

    \item $d(u_3)$ edges incident to $u_1$ (recall that there are at least $4$) will be labelled with $\{m-1,m-5,m-9,\ldots, m-4(d(u_3)-1)-1\}$.

    \item $d(u_3)$ edges incident to $u_2$ will be labelled with $\{m-2,m-6,m-10,\ldots,m-4(d(u_3)-1)-2\}$.

    \item The $d(u_3)$ edges incident to $u_3$ will be labelled with $\{m-3,m-7,\ldots,m-4(d(u_3)-1)-3\}$.

    \item The other edges in $E_2$ will be labelled with the unreserved labels.
\end{itemize}

Note that since $m \geq 7n$, it is always possible to reserve these labels.

We additionally want to guarantee that $\sigma(r)$ is maximal in $G$ once the labelling is done. To this end, one can first notice that the reservation of labels $\{m,m-4,\ldots,m-4(n-5)\}$ for the edges in $E_1$ induces the creation of intervals $I_1 = \{m-1,m-2,m-3\}$, $I_2 = \{m-5,m-6,m-7\}, \ldots, I_{n-5} = \{m-4(n-6) - 1, m-4(n-6) - 2, m-4(n-6) - 3\}$. We will assign the corresponding labels such that each vertex $v \in V \setminus \{r\}$ has at most one incident edge labelled in each $I_j$.

In order to do this, let us consider the bipartite graph $G_1 = (V_{G_1},E_{G_1})$ defined as follow:

\begin{itemize}
    \item $V_{G_1} = V \setminus \{r\}$,

    \item $u_3$ has all its $d(u_3)$ incident edges in $E_{G_1}$,

    \item We arbitrarily select $d(u_3)$ edges incident to $u_1$ and $d(u_3)$ edges incident to $u_2$, that we put in $E_{G_1}$.
\end{itemize}

$G_1$ is well defined since $d(u_1) \geq d(u_2) \geq d(u_3)$. We have $|E_{G_1}| = 3d(u_3)$. Moreover, $\Delta(G_1) = d(u_3)$ since every vertex $v \in H$ is such that $d_{G_1}(v) \leq 3$ and $d'(u_3) = d(u_3) \geq 4$. Hence, Theorem 6.1.5 in \cite{west} (originally proven by König) shows that there exists a proper colouring of the edges of $G_1$ with $d(u_3)$ colours. Since $d_{G_1}(u_1) = d_{G_1}(u_2) = d_{G_1}(u_3) = d(u_3)$, this colouring is made of colour classes $C_1,C_2,\ldots, C_{d(u_3)}$ such that each $C_i$ contains exactly $3$ edges.

For $1 \leq j \leq d(u_3)$, we associate to the three edges of $C_j$ the interval $I_j = \{m-4(j-1)-1,m-4(j-1)-2,m-4(j-1)-3\}$, such that $m - 4(j-1) - 1$ is assigned to the edge incident to $u_1$ in $C_j$, $m-4(j-1)-2$ is assigned to the edge incident to $u_2$ in $C_j$, and $m-4(j-1)-3$ is assigned to the edge incident to $u_3$ in $C_j$.

Let us now consider the graph $G_2 = (V \setminus \{r,u_3\},E_{G_2})$, where $E_{G_2} = E_2 \setminus E_{G_1}$ is the set of edges $e \not\in E_{G_1}$, and $e$ is not incident to $r$.

Since $\Delta(G_2) \leq n - 5$, thanks to Vizing's theorem (Theorem 6.1.7 in \cite{west}), there exists a proper edge-colouring of $G_2$ with at most $\Delta(G_2) + 1 \leq n - 4$ colours. Let $C'_1, C'_2, \ldots, C'_{n-4}$ be the colour classes for this colouring. Since $m_{G_2} = m - 3d(u_3) - (n-4) \geq 3n$ (since $m \geq 7n$), we can assume that there are at least $3$ edges in each class. Indeed, let us suppose that there exists a class $C'_i$ such that $|C'_i| \leq 2$. Then there exists a class $C'_j$ such that $|C'_j| > 3$. Let $H_{i,j}$ be the partial graph of $H$ obtained by taking only the edges with colours $i$ and $j$. $H_{i,j}$ is a collection of even-length alternating cycles and alternating paths. Since $|C'_j| > |C'_i|$, there exists at least a path $P_{i,j}$ such that its first and last edges have colour $j$ (the two edges might be the same). By exchanging the colours $i$ and $j$ in $P_{i,j}$, we increase $|C'_i|$ by $1$, and we can repeat the process until $|C'_i| \geq 3$ for each colour class $C'_i$. We can then assume in the following that there are at least $3$ edges in every class. Finally, we sort the classes such that the $d(u_1) - d(u_3)$ edges incident to $u_1$ in $G_2$ (recall that $d(u_1) \geq d(u_3)$) are in the classes $C'_1, C'_2, \ldots, C'_{d(u_1) - d(u_3)}$.

For $d(u_3) + 1 \leq j \leq d(u_1)$, we associate three edges of the class $C'_{j-d(u_3)}$ to the interval $I_j$, such that $m - 4(j-1) - 1$, i.e. the largest label in $I_j$, is assigned to the edge incident to $u_1$. The other two labels in $I_j$ are arbitrarily assigned to two other edges in the class.

For $d(u_1) + 1 \leq j \leq n - 5$, we associate three edges of the class $C'_{j - d(u_3)}$ to the interval $I_j$.

Once every label of the intervals $I_j$, for $1 \leq j \leq n - 5$, has been assigned to an edge, we arbitrarily label the remaining edges of $E_{G_2}$ (note that these edges have either both endpoints in $H$, or one endpoint in $H$ and the other endpoint is $u_2$). We then sort the vertices of $H$ by increasing order of $\sigma'$: $\sigma'(v_1) \leq \sigma'(v_2) \leq \ldots \leq \sigma'(v_{n-4})$, similarly to the proof of Theorem \ref{n-1}. We then label the $rv_i$ edges in this order, using the labels reserved for $E_1$ in increasing order as well. This concludes the first stage of the labelling of $G$.

We now aim to show some properties of our labelling. First, we have $\sigma(u_2) \geq \sigma(u_3) + 4$ due to the labelling of $G_1$ and due to $d(u_2) \geq d(u_3) \geq 4$. Moreover, $\sigma(u_1) \geq \sigma(u_2) + 4$ due to the labelling of $G_1$ and $G_2$. Overall, we obtain:

\begin{prop}
    \label{ecartu}
    $\sigma(u_3) + 4 \leq \sigma(u_2)$ and $\sigma(u_2) + 4 \leq \sigma(u_1)$.
\end{prop}

Furthermore, every vertex $v \in V \setminus \{r\}$ has at most one incident edge labelled in every interval $I_j$, $1 \leq j \leq n - 5$, since the three labels of $I_j$ were assigned to three edges in the same colour class. Moreover, recall that $v$ cannot be a neighbour of all three $u_1,u_2,u_3$; in other words, $d_{G_1}(v) \leq 2$. This means that, for any $v \in H$:

\begin{align*}
        \sigma(v) \leq m &+ (m-1) + (m - 6) + (m-17) + (m-21) \\
        & +\cdots + (m - 4(n-6) - 1) + (m - 4(n-5) - 1) + (m - 4(n-5)-2)
\end{align*}

The $m$ term is an upper bound of the label of $rv$. The $(m-1) + (m-6)$ term is an upper bound of the sum of labels that can be assigned to edges incident to $v$ during the labelling of $G_1$. $G_1$ has at least $12$ edges since $d(u_3) \geq 4$, hence the largest label that can be assigned to an edge incident to $v$ during the labelling of $G_2$ is $m - 17$ (since $d(r) = n - 4 \geq 12$).

We also have that:

\begin{align*}
        \sigma(r) & = m + (m-4) + (m-8) + (m - 12) + (m-16) + (m-20) \\
        &+ \cdots + (m-4(n-6)) + (m - 4(n-5))\\
        \sigma(u_1) &\leq (m-1) + (m-5) + \cdots + (m - 4(n-5) - 1) \\
        \sigma(u_3) & < \sigma(u_2) < \sigma(u_1)
\end{align*}

Since $n \geq 16$, we obtain that:

\begin{prop}
    \label{ecartracine}
    For each vertex $x \in V \setminus \{r\}$, $\sigma(r) \geq \sigma(x) + 4$.
\end{prop}

Furthermore, since we have reserved labels spaced by 4 for edges of $E_1$, and since we sorted vertices in $H$ in increasing order of $\sigma'$ before labelling those edges, we obtain the following property:

\begin{prop}
    \label{ecartv}
    For each $v_i \in H, 1 \leq i \leq n - 5$, $\sigma(v_i) + 4 \leq \sigma(v_{i+1})$.
\end{prop}

These three properties will be used in Section 3, where we will exchange some labels to obtain an antimagic labelling. Note that thanks to Property \ref{ecartracine}, we will ignore in the following all changes made to the value of $\sigma(r)$, as the property $\sigma(r) > \sigma(v)$ for any $v \neq r$ still holds true afterwards. We will justify this hypothesis at the end of Section 3.

An illustration of the labelling of $G$ is shown in Figure \ref{etiquetageinitial}. For every $i$, we will call $y_i$ the neighbour of $u_j$ (for some $j$) such that the edge $u_jy_i$ is labelled by $m - i$. Notice that, for instance, $y_1$ and $y_6$ can be the same vertex in Figure \ref{etiquetageinitial}. Moreover, for every $i$, we will call $w_i$ the neighbour of $r$ such that $rw_i$ is labelled by $m - i$. Let us call $Y$ the set of vertices $y_i$ and $W$ the set of vertices $w_i$. We have $Y \subseteq W = \{v_1,v_2,\ldots,v_{n-4}\} = V_H$.

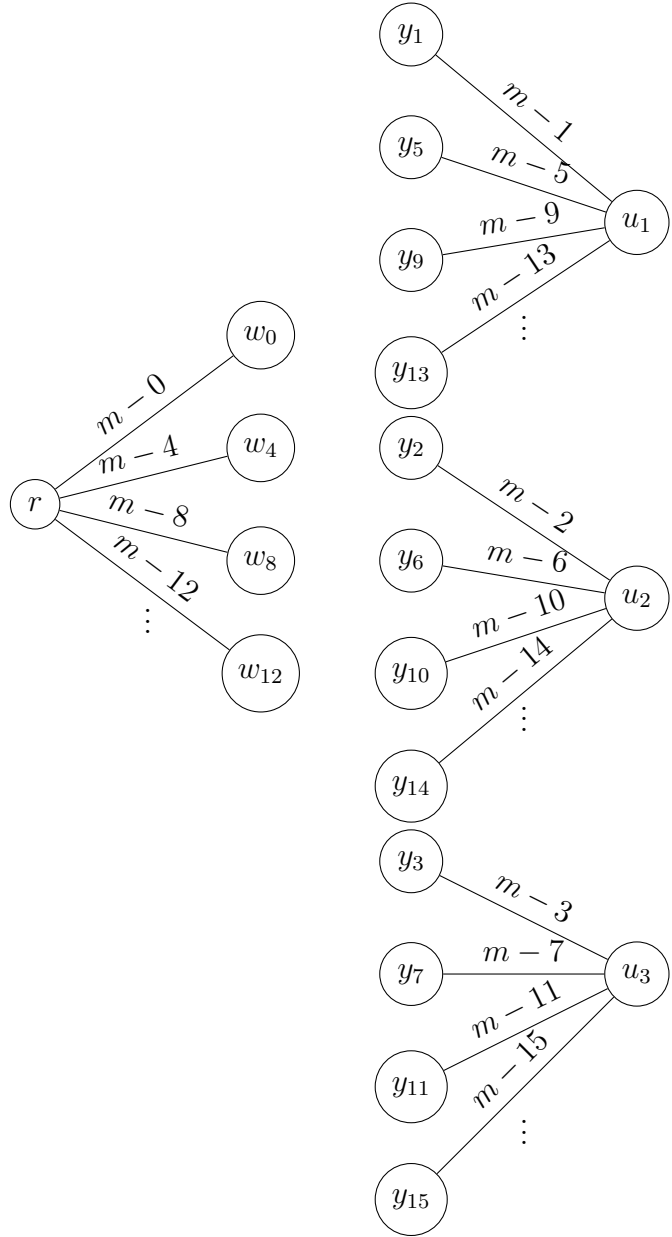
\begin{figure}[h!]
        \centering
        \begin{tikzpicture}[every circle node/.style = {draw}, every edge/.style={draw,sloped}]
            \node[circle] at (-1,0.75) (r) {$r$};
            \node[circle] at (2,3) (w0) {$w_0$};
            \node[circle] at (2,1.5) (w4) {$w_4$};
            \node[circle] at (2,0) (w8) {$w_8$};
            \node[circle] at (2,-1.5) (w12) {$w_{12}$};
            \node at (0.5,-0.7) {$\vdots$};

            \node[circle] at (4,7) (y1) {$y_1$};
            \node[circle] at (4,5.5) (y5) {$y_5$};
            \node[circle] at (4,4) (y9) {$y_9$};
            \node[circle] at (4,2.5) (y13) {$y_{13}$};
            \node[circle] at (4,1.5) (y2) {$y_2$};
            \node[circle] at (4,0) (y6) {$y_6$};
            \node[circle] at (4,-1.5) (y10) {$y_{10}$};
            \node[circle] at (4,-3) (y14) {$y_{14}$};
            \node[circle] at (4,-4) (y3) {$y_3$};
            \node[circle] at (4,-5.5) (y7) {$y_7$};
            \node[circle] at (4,-7) (y11) {$y_{11}$};
            \node[circle] at (4,-8.5) (y15) {$y_{15}$};

            \node at (5.5,-7.5) {$\vdots$};
            \node at (5.5,-2) {$\vdots$};
            \node at (5.5,3.2) {$\vdots$};

            \node[circle] at (7,4.5) (u1) {$u_1$};
            \node[circle] at (7,-0.5) (u2) {$u_2$};
            \node[circle] at (7,-5.5) (u3) {$u_3$};

            \foreach \x in {0,4,8,12}
                \draw (r) edge["$m - \x$"] (w\x);
            \foreach \x in {1,5,9,13}
                \draw (u1) edge["$m - \x$"] (y\x);
            \foreach \x in {2,6,10,14}
                \draw (u2) edge["$m - \x$"] (y\x);
            \foreach \x in {3,7,11,15}
                \draw (u3) edge["$m - \x$"] (y\x);
        \end{tikzpicture}
        \caption{Illustration of the labelling of $G$.}
        \label{etiquetageinitial}
    \end{figure}

These notations (meaning the use of $y,w$ or $v$ to denote the vertices of $H$) allow us in the following to give a simple way to identify which labels are being exchanged, and to distinguish between seeing vertices of $H$ as 'neighbours of $u_i$' or as 'neighbours of $r$'.

With our notations, two vertices $y_i$ and $y_{i'}$ that are neighbours of the same vertex $u_j$ are necessarily distinct. Similarly, two vertices $w_i$ and $w_{i'}$ are also distinct. Moreover, with our construction, the three vertices $y_i$ that correspond to the same interval $I_j = \{m-4(j-1)-1,m-4(j-1)-2,m-4(j-1)-3\}$ are all distinct (since we labelled three edges in the same colour class with those three labels). For instance, $y_1,y_2$ and $y_3$ are three distinct vertices.

We formally define the notion of conflict between two vertices:

\begin{definition}
    Two vertices $u$ and $u'$ are said to be in \emph{conflict}, written \cf{u}{u'}, if $\sigma(u) = \sigma(u')$.
\end{definition}

A labelling is then antimagic if and only if there are no conflicts in the graph. Note that, with our labelling process, the only conflicts that could arise in the graph are, thanks to Properties \ref{ecartu}, \ref{ecartracine} and \ref{ecartv}, between $u_1,u_2$ or $u_3$ on the one hand, and some vertex $v \in H$ on the other hand.

\section{Conflicts resolution}

In this section, we explain how to solve these potential conflicts after the initial labelling is done, in order to obtain a conflict-free labelling, meaning an antimagic labelling.

We summarize the possible exchanges of labels and their consequences on the different values of $\sigma$ in Table \ref{tab1}. Notice that, for any exchange $\lambda_i$ or $\gamma_i$, the involved vertices $y_i$ and $y_{i+1}$ are necessarily distinct. In all the following, $v_1$ (respectively $v_2$, $v_3$) will denote the vertex of $H$ with the value of $\sigma$ closest to the one of $u_1$ (respectively of $u_2,u_3$) - in case of a tie, we pick one vertex arbitrarily.

\begin{table}[h!]
        \centering
        \begin{tabular}{|c|c|c|c|}
            \hline
             Notation & Exchanges & Evolution of $\sigma(u_i)$ & Evolution of $\sigma(v_i)$ \\
             \hline
             $\lambda_{1}$ & $m-1 \leftrightarrow m-2$ & \multirow{4}{*}{$\sigma(u_1) - 1, \sigma(u_2) + 1$} & $\sigma(y_1) - 1, \sigma(y_2) + 1$\\
             $\lambda_{5}$ & $m-5 \leftrightarrow m-6$ & & $\sigma(y_5) - 1, \sigma(y_6) + 1$ \\
             $\lambda_{9}$ & $m-9\leftrightarrow m-10$ & & $\sigma(y_9) - 1, \sigma(y_{10}) + 1$ \\
             $ \lambda_{13}$ & $m - 13 \leftrightarrow m - 14$ & & $\sigma(y_{13}) - 1$, $\sigma(y_{14}) + 1$ \\
             \hline
             $\gamma_{2}$ & $m-2 \leftrightarrow m-3$ & \multirow{4}{*}{$\sigma(u_2) - 1, \sigma(u_3) + 1$} & $\sigma(y_2) - 1, \sigma(y_3) + 1$ \\
             $\gamma_{6}$ & $m - 6 \leftrightarrow m - 7$ & & $\sigma(y_6) - 1, \sigma(y_7) + 1$ \\
             $\gamma_{10}$ & $m - 10 \leftrightarrow m - 11$ & & $\sigma(y_{10}) - 1, \sigma(y_{11}) + 1$ \\
             $\gamma_{14}$ & $m - 14 \leftrightarrow m - 15$ & & $\sigma(y_{14}) - 1, \sigma(y_{15}) + 1$ \\
             \hline
             $\mu_{0}$ & $m \leftrightarrow m-1$ & \multirow{4}{*}{$\sigma(u_1) + 1$} & $\sigma(y_1) + 1, \sigma(w_0) - 1$ \\
             $ \mu_{4}$ & $m-4 \leftrightarrow m-5$ & & $\sigma(y_5) + 1, \sigma(w_4) - 1$ \\
             $ \mu_{8}$ & $m-8\leftrightarrow m-9$ & & $\sigma(y_9) + 1, \sigma(w_8) - 1$ \\
             $ \mu_{12}$ & $m-12 \leftrightarrow m - 13$ & & $\sigma(y_{13}) + 1, \sigma(w_{12}) - 1$  \\
             \hline
             $\rho_{3}$ & $m-3 \leftrightarrow m-4$ & \multirow{4}{*}{$\sigma(u_3) - 1$} & $\sigma(y_3) - 1, \sigma(w_4) + 1$ \\
             $\rho_{7}$ & $m-7 \leftrightarrow m-8$ & & $\sigma(y_7) - 1, \sigma(w_8) + 1$ \\
             $\rho_{11}$ & $m-11 \leftrightarrow m-12$ & & $\sigma(y_{11}) - 1, \sigma(w_{12}) + 1$ \\
             $ \rho_{15}$ & $m - 15 \leftrightarrow m - 16$ & & $\sigma(y_{15}) - 1, \sigma(w_{16}) + 1$ \\
             \hline
        \end{tabular}
        \caption{Summary of the possible exchanges (the modifications of the value of $\sigma(r)$ are ignored).}
        \label{tab1}
\end{table}

We explain how to obtain an antimagic labelling in every case:

\begin{itemize}
    \item \textbf{Case 1: \cf{u_1}{v_1} and \cf{u_2}{v_2} and \cf{u_3}{v_3}.}

    The $\lambda_i$ exchange solves the conflicts over $u_1$ and $u_2$, meaning that afterwards, the values of $\sigma(u_1)$ and $\sigma(u_2)$ are different from all the other values of $\sigma(x), x \in V$, unless $v_1 = y_i$ or $v_2 = y_{i+1}$, thanks to Property \ref{ecartv}. There is at most one $\lambda_i$ exchange such that $y_i = v_1$ and at most one $\lambda_i$ exchange such that $y_{i+1} = v_2$ (since all $y_i$'s are distinct because they are all neighbours of $u_1$, and all $y_{i+1}$'s are distinct because they are all neighbours of $u_2$). Thus, there are necessarily at least two of the four $\lambda_i$ exchanges such that $y_i \neq v_1$ and $y_{i+1} \neq v_2$, guaranteeing that $u_1$ and $u_2$ are conflict-free after the exchange. Let us call them $\lambda_i$ and $\lambda_{i'}$, and consider first that $\lambda_i$ has been applied.

    After the exchange, the value of $\sigma(u_1)$ has been decreased by $1$, and the value of $\sigma(u_2)$ has been increased by $1$. Note that, thanks to Property \ref{ecartu}, those two values cannot be equal to each other.

    Notice that, if the value of $\sigma(v_3)$ was modified during the exchange $\lambda_i$ (i.e. $v_3 = y_i$ or $v_3 = y_{i+1}$), then $u_3$ is not in conflict with $v_3$ anymore, and hence the labelling is conflict-free and antimagic.

    We will now perform a second exchange, $\rho_k$, between labels over an edge incident to $u_3$ and another edge incident to $r$, in order to solve the conflict over $u_3$. We will need to make sure we are not creating a new conflict involving $u_1$ or $u_2$ in the meanwhile.

    The $\rho_k$ exchange yields an antimagic labelling, unless:

    \begin{itemize}
        \item There is a conflict between two different $u_i$; however, thanks to Property \ref{ecartu}, this is impossible.

        \item There is still a conflict between a vertex $v \in H$ and $u_1$. After the two exchanges, the value of $\sigma(u_1)$ has been decreased by 1, and the variation of the value of any $\sigma(v)$ is between -2 and +2. Due to Property \ref{ecartv}, $\sigma(v) = \sigma(u_1)$ implies that $v = v_1$, and that the value of $\sigma(v_1)$ has been decreased by 1 after the two exchanges. Hence, the only two possibilities are $v = v_1 = y_i$ and $v = v_1 = y_k$. Since $\lambda_i$ was chosen such that $v_1 \neq y_i$, this implies $v_1 = y_k$.

        \item There is still a conflict between a vertex $v \in H$ and $u_2$. After the two exchanges, the value of $\sigma(u_2)$ has been increased by 1, and the variation of the value of any $\sigma(v)$ is between -2 and +2. Due to Property \ref{ecartv}, $\sigma(v) = \sigma(u_2)$ implies that $v = v_2$, and that the value of $\sigma(v_2)$ has been increased by 1 after the two exchanges. Hence, the only two possibilities are $v = v_2 = y_{i+1}$ and $v = v_2 = w_{k+1}$. Since $\lambda_i$ was chosen such that $v_2 \neq y_{i+1}$, this implies $v_2 = w_{k+1}$.

        \item There is still a conflict between a vertex $v \in H$ and $u_3$. After the two exchanges, the value of $\sigma(u_3)$ has been decreased by $1$, and the variation of the value of any $\sigma(v)$ is between -2 and +2. Due to Property \ref{ecartv}, $\sigma(v) = \sigma(u_3)$ implies that $v = v_3$, and that the value of $\sigma(v_3)$ has been decreased by 1 after the two exchanges. Hence, the only two possibilities are $v = v_3 = y_i$ and $v = v_3 = y_k$. However, $v_3 \neq y_i$; otherwise, as previously explained, the graph was already proven antimagic after the $\lambda_i$ exchange. This means that $v_3 = y_k$.

        \item There is a conflict between two vertices of $H$. Since we are exactly making two increases of +1 and two decreases of -1 over the values of $\sigma$, and due to Property \ref{ecartv}, the only case where a conflict could arise is where the three following equalities are verified: $y_i = y_k$ (meaning the value of $\sigma(y_i)$ was decreased by $2$), $y_{i+1} = w_{k+1}$ (meaning the value of $\sigma(y_{i+1})$ was increased by $2$), and $\sigma(y_i) = \sigma(y_{i+1}) + 4$.
    \end{itemize}

    This means that, if no $\rho_k$ exchange allows us to obtain an antimagic labelling, this means that there exist four distinct indices $k_1,k_2,k_3,k_4 \in \{3,7,11,15\}$, such that the following equalities are all true:

    \begin{itemize}
        \item $v_1 = y_{k_1}$,
        \item $v_2 = w_{k_2 + 1}$,
        \item $v_3 = y_{k_3}$,
        \item $y_i = y_{k_4}$ and $y_{i+1} = w_{k_4+1}$ and $\sigma(y_i) = \sigma(y_{i+1}) + 4$.
    \end{itemize}

    Indeed, if some $k_h \in \{3,7,11,15\}$ is such that all the above equalities are false, then the $\rho_{k_h}$ exchange yields a conflict-free labelling, i.e. an antimagic one. Moreover, notice that each of these equalities can be verified by at most one pair $(y_{k_h}, w_{k_h + 1})$, with $k_h\in \{3,7,11,15\}$, since two vertices $w_p$ and $w_q$ are distinct, as are two vertices $y_p$ and $y_q$ that are neighbours of the same $u_s$.

    If this is the case, we can then perform the exchange $\lambda_{i'}$ instead of $\lambda_i$. We obtain that $\rho_{k_4}$ is such that:

    \begin{itemize}
        \item $v_1 = y_{k_1} \neq y_{k_4}$,
        \item $v_2 = w_{k_2 + 1} \neq w_{k_4 + 1}$,
        \item $v_3 = y_{k_3} \neq y_{k_4}$,
        \item $y_{i'} \neq y_i = y_{k_4}$ and $y_{i'+1}\neq y_{i+1} = w_{k_4+1}$.
    \end{itemize}

    We can then obtain an antimagic labelling by performing $\rho_{k_4}$.
    
    \item \textbf{Case 2: \cf{u_1}{v_1} and \cf{u_2}{v_2}.}

    We will perform a $\lambda_i$ exchange. Recall that the $u_i$ cannot be in conflict with each other once the exchange is applied, due to Property \ref{ecartu}. Moreover, two vertices of $H$ cannot be in conflict with each other as well due to Property \ref{ecartv}.

    The $\lambda_i$ exchange then yields an antimagic labelling unless $y_i = v_1, y_{i+1} = v_2, \sigma(u_3) = \sigma(y_i) - 1$ or $\sigma(u_3) = \sigma(y_{i+1}) + 1$.

    However, those last two possibilities are mutually exclusive due to Property \ref{ecartv}, meaning there is at most one $\lambda_i$ exchange such that $\sigma(u_3) = \sigma(y_i) - 1$ or $\sigma(u_3) = \sigma(y_{i+1}) + 1$. We then have four possible $\lambda_i$ exchanges and three possible issues, and therefore there always exists an exchange yielding an antimagic labelling.

    \item  \textbf{Case 3: \cf{u_2}{v_2} and \cf{u_3}{v_3}.}

    The proof in this case is analogous to the previous case, with the exchanges $\gamma$  being considered instead of $\lambda$.

    \item \textbf{Case 4: \cf{u_1}{v_1} and \cf{u_3}{v_3}.}

    Contrary to Case 1, we cannot simply exchange labels between edges incident to $u_1$ and $u_3$, as this would induce changes of +2 and -2 on the values of $\sigma(v), v \in H$, and conflicts between vertices of $H$ could then arise.

    Let us suppose first that $|\sigma(v_2) - \sigma(u_2)| \geq 2$. Since there is at most one $\lambda_i$ exchange such that $y_i = v_1$, and at most one $\lambda_i$ exchange such that $y_i = v_2$, there are at least two possible exchanges $\lambda_i$ and $\lambda_{i'}$ such that $v_1 \neq y_i,y_{i'}$ and $v_2 \neq y_i,y_{i'}$, meaning that both $u_1$ and $u_2$ are conflict-free once any of these two exchanges is applied. Let us perform $\lambda_i$ for now; the value of $\sigma(u_2)$ is increased by 1. It is then possible that $v_2$ is such that $\sigma(u_2) + 1 = \sigma(v_2) - 1$. Similarly as in Case 1, we assume $v_3 \neq y_i$ (otherwise $u_3$ and $v_3$ are not in conflict anymore after the $\lambda_i$ exchange), and the $\rho_k$ exchange then yields an antimagic labelling, unless:

    \begin{itemize}
        \item $v_1 = y_k$,
        \item $v_2 = y_k$ (in the case $\sigma(u_2) + 1 = \sigma(v_2) - 1$),
        \item $v_3 = y_k$,
        \item $y_i = y_k$ and $w_{k+1} = y_{i+1}$ and $\sigma(y_i) = \sigma(y_{i+1}) + 4$.
    \end{itemize}

    If no $\rho_k$ yields an antimagic labelling, we can, as in Case 1, consider the $\lambda_{i'}$ exchange instead of $\lambda_i$. We can then guarantee that $y_k \neq y_{i'}$ and $w_{k+1} \neq y_{i'+1}$, and obtain an antimagic labelling.

    Let us now suppose that $|\sigma(v_2) - \sigma(u_2)| = 1$. Suppose first that $\sigma(u_2) = \sigma(v_2) + 1$. If there exists $j$ such that:

    \begin{itemize}
        \item $v_1 = w_{j+1}$,
        \item $y_ju_3 \in E$ and $y_j \neq v_3$,
        \item $w_{j+1} \neq y_j$.
    \end{itemize}

    Then we perform the $\rho_j$ exchange. The modified values of $\sigma$ are the following: $\sigma(u_3) - 1, \sigma(y_j) - 1, \sigma(w_{j+1}) + 1$.

    Since $y_j \neq v_3$, due to Property \ref{ecartv}, $u_3$ is conflict-free. Since $v_1 = w_{j+1}$ and $w_{j+1} \neq y_j$, $u_1$ is also conflict-free. Finally, since $v_1 = w_{j+1}$, $v_2 \neq w_{j+1}$ due to Property \ref{ecartu}. This means that $u_2$ is conflict-free as well, and hence the labelling is antimagic.

    If such a $j$ does not exist, let us consider the $\mu_i$ exchanges. For each $i$, the $\mu_i$ exchange (strictly) decreases the number of conflicts, unless $v_1 = y_{i+1}$ or $v_2 = y_{i+1}$. Therefore there exist at least two $\mu$ exchanges such that $u_1$ and $u_2$ are conflict-free once the exchange is applied. Let $\mu_i$ be such an exchange.

    If $u_3$ is still in conflict once $\mu_i$ is performed (meaning that the value of $\sigma(v_3)$ was unchanged by $\mu_i$), we now consider the $\rho$ exchanges.

    Let us consider $\rho_j$. This yields an antimagic labelling, unless $v_3 = y_j$ or $v_2 = w_{j+1}$ or ($y_{i+1} = w_{j+1}$ and $w_i = y_j$ and $\sigma(w_i) = \sigma(w_j) + 4$) or ($w_{j+1} = v_1 \neq y_{i+1}$ and $w_{j+1} \neq y_j$) (meaning that, in this last case, after the two exchanges are performed, the value of $\sigma(v_1)$ was increased by $1$).

    However, this last case is impossible since we assumed such a $j$ did not exist. Therefore there are at most three issues and four possible exchanges (each issue can be associated with at most one $\rho_j$ exchange), hence there is at least one $\rho$ exchange yielding an antimagic labelling.

    The reasoning is easily adaptable to the case where $\sigma(u_2) = \sigma(v_2) - 1$; the equalities to exclude in this case will be $v_2 = w_i$ instead of $v_2 = y_{i+1}$ and then $v_2 = y_j$ instead of $v_2 = w_{j+1}$.

    \item \textbf{Case 5: \cf{u_1}{v_1}.}

    The $\mu_k$ exchange yields an antimagic labelling, thanks to Property \ref{ecartu}, unless:

    \begin{itemize}
        \item $v_1 = y_{k+1}$,
        \item $v_2 = y_{k+1}$ (and $\sigma(v_2) + 1 = \sigma(u_2)$) or $v_2 = w_k$ (and $\sigma(v_2) - 1 = \sigma(u_2)$); these two cases are mutually exclusive due to Property \ref{ecartv},
        \item $v_3 = y_{k+1}$ (and $\sigma(v_3) + 1 = \sigma(u_3)$) or $v_3 = w_k$ (and $\sigma(v_3) - 1 = \sigma(u_3)$); these two cases are mutually exclusive due to Property \ref{ecartv}.
    \end{itemize}

    There are three possible issues and four possible $\mu_k$ exchanges (each issue can be associated with at most one $\mu_k$ exchange), hence there always exists an exchange yielding an antimagic labelling.

    \item \textbf{Case 6: \cf{u_3}{v_3}}

    The reasoning is analogous to the previous case; a $\rho_k$ exchange yields an antimagic labelling, unless:

    \begin{itemize}
        \item $v_1 = y_k$ or $v_1 = w_{k+1}$, but these two cases are mutually exclusive;
        \item $v_2 = y_k$ or $v_2 = w_{k+1}$, but these two cases are mutually exclusive;
        \item $v_3 = y_k$.
    \end{itemize}

    Again, there are three possible issues and four $\rho_k$ exchanges, therefore there always exists an exchange yielding an antimagic labelling.

    \item \textbf{Case 7: \cf{u_2}{v_2}.}

    \begin{itemize}
        \item \textbf{Case 7.1:} $|\sigma(v_1) - \sigma(u_1)| \geq 2$

        The $\lambda_i$ exchange yields an antimagic labelling unless $v_2 = y_{i+1}$, $v_1 = y_{i+1}$ (and $\sigma(v_1) = \sigma(u_1) - 2$), $\sigma(u_3) = \sigma(y_i) - 1$ (and $v_3 = y_i$) or $\sigma(u_3) = \sigma(y_{i+1}) + 1$ (and $v_3 = y_{i+1}$). As in the previous case, those last two equalities are mutually exclusive, due to Property \ref{ecartv}; overall, there are three possible issues and four possible exchanges, hence there exists a $\lambda_i$ yielding an antimagic labelling.

        \item \textbf{Case 7.2:} $\sigma(v_1) = \sigma(u_1) + 1$.

        We can apply the same reasoning as in Case 7.1, because the $\lambda$ exchanges all decrease the value of $\sigma(u_1)$ by $1$ (simply note that, due to Property \ref{ecartv}, there is one less issue since $\sigma(v_1) \neq \sigma(u_1) - 2$).

        \item \textbf{Case 7.3:} $|\sigma(v_3) - \sigma(u_3)| \geq 2$.

        The reasoning is the same as in Case 7.1, by using the $\gamma$ exchanges instead of $\lambda$. The potential issues arise when $v_2 = y_i$, $v_3 = y_i$ (and $\sigma(v_3) = \sigma(u_3) + 2$), $\sigma(u_1) = \sigma(y_i) - 1$ (and $y_i = v_1$), or $\sigma(u_1) = \sigma(y_{i+1}) + 1$ (and $y_{i+1} = v_1$). Again, there are three potential issues and four possible exchanges, therefore there exists a $\gamma_i$ yielding an antimagic labelling.

        \item \textbf{Case 7.4:} $\sigma(v_3) = \sigma(u_3) - 1$.

        The same reasoning applies again (with only two possible issues and four possible exchanges), as the $\gamma$ exchanges all increase the value of $\sigma(u_3)$ by $1$.

        \item \textbf{Case 7.5:} $\sigma(v_1) = \sigma(u_1) - 1$ and $\sigma(v_3) = \sigma(u_3) + 1$.

        If there exists $\mu_j$ such that $y_{j+1} = v_3$ (hence $y_{j+1} \neq v_1$ due to Property \ref{ecartu}), we perform this exchange: this increases $\sigma(u_1)$ by 1, and $\sigma(u_1) + 1 = \sigma(v_1) + 2$. This also increases $\sigma(v_3)$ by 1, and $\sigma(v_3) + 1 = \sigma(u_3) + 2$. Therefore, once the $\mu_j$ exchange is performed, there is a difference of at least 2 between the values of $\sigma(u_1)$ (respectively $\sigma(u_3)$) and $\sigma(v_1)$ (respectively $\sigma(v_3)$). If the value of $\sigma(v_2)$ is modified with this exchange, there are no more conflicts in the graph, due to Property \ref{ecartv}; hence we will assume it is unchanged.

        Let us then consider the $\lambda$ exchanges, excluding $\lambda_{j+1}$; the $\lambda_i$ exchange yields the following changes (after the two exchanges): $\sigma(u_2) + 1, \sigma(v_3) + 1, \sigma(w_j) - 1, \sigma(y_i) - 1, \sigma(y_{i+1}) + 1$. We then obtain an antimagic labelling, unless $v_1 = y_{i+1}$ or $v_2 = y_{i+1}$ (since we excluded $\lambda_{j+1}$, $y_i \neq y_{j+1} = v_3$). There are two possible issues and three possible exchanges (each issue can be associated with at most one $\lambda_i$ exchange), so we can obtain an antimagic labelling.

        Let us now consider that there are no $\mu_j$ such that $y_{j+1} = v_3$. Since there are four possible $\mu_j$, there exist at least two such that:

        \begin{itemize}
            \item $w_j \neq v_3$, meaning that $u_3$ is conflict-free after the exchange,

            \item $y_{j+1} \neq v_1$, meaning that the difference between the values of $\sigma(u_1)$ and $\sigma(v_1)$ is at least 2 once $\mu_j$ has been performed, and $u_1$ is conflict-free after the exchange.
        \end{itemize}

        As in the previous case, if the value of $\sigma(v_2)$ is modified with the $\mu_j$ exchange, then the labelling is conflict-free thus antimagic; we will then assume that $\sigma(v_2)$ is unchanged.

        Let us now consider the $\lambda$ exchanges, excluding $\lambda_{j+1}$. The $\lambda_i$ exchange yields an antimagic labelling, unless:

        \begin{itemize}
            \item $v_2 = y_{i+1}$ (since $\sigma(v_2)$ is unchanged after the $\mu_j$ exchange),
            \item $v_1 = y_{i+1}$ (since $v_1 \neq y_{j+1}$ due to the choice of $\mu_j$),
            \item $v_3 = y_i$ (since $v_3 \neq w_j$ due to the choice of $\mu_j$).
        \end{itemize}

        However, this last equality is impossible because we could have otherwise performed the $\mu_{i-1}$ exchange, in which we would have had $v_3 = y_i$, but we previously assumed no such exchange existed. Therefore, we have 2 possible issues and 3 possible exchanges (each issue can be associated with at most one $\lambda_i$ exchange), it is then possible to obtain an antimagic labelling.
    \end{itemize}

\end{itemize}

In every case, the value of $\sigma(r)$ is decreased at most by 1. Meanwhile, a given vertex $v \neq r$ has its value $\sigma(v)$ increased at most by 2. Due to Property \ref{ecartracine}, we obtain in every case that $\sigma(r) > \sigma(v)$ once the exchanges are applied, for any $v \neq r$, thus confirming the fact that we could ignore the changes made to the value of $\sigma(r)$.

\section{Labelling process for the other cases}

Let us now consider that the hypothesis made on the structure of the graph at the beginning of Section $2$ do not hold. First, if $\{u_1,u_2,u_3\}$ does not induce an independent set, we simply label the edges between two $u_i$ with the smallest labels - \emph{i.e.} $1,2$ and $3$ if necessary -, sorting the edges by inverted lexicographic order. Hence, for instance, in the case where $\{u_1,u_2,u_3\}$ induces a clique, $u_2u_3$ is labelled by $1$, $u_1u_3$ by $2$ and $u_1u_2$ by $3$. We then proceed to the labelling as described in Section 2, and it is easy to verify that Properties \ref{ecartu}, \ref{ecartracine} and \ref{ecartv} are not modified, and neither is Section 3, proving that we can obtain an antimagic labelling.

Second, assume now that $d'(u_3) \leq 3$. Recall that $d'(u_1) \geq d'(u_2) \geq d'(u_3)$. Let $i = \min \{j \in \{1,2,3\} | d'(u_j) \leq 3\}$.

We will label edges incident to $\{u_j, j\geq i\}$ with the smallest labels (excluding the ones used to label edges between two $u_k$'s), in order to guarantee that, for any $v \in H$, $\sigma(u_3) < \ldots < \sigma(u_i) < \sigma(v)$.

We distinguish three cases depending on the value of $i$:

\begin{itemize}
    \item $i = 1$.

    We label the graph by reserving the $n - 4$ largest labels for the edges incident to $r$, and using the smallest labels for edges incident to $u_1,u_2$ and $u_3$, dealing with the $u_k$ in decreasing order (meaning we label first all the edges incident to $u_3$, then the edges incident to $u_2$, and finally the edges incident to $u_1$). We give an illustration of the labelling of the graph in Figure \ref{illus2}. Recall that, before labelling the edges incident to $r$, we sort the vertices of $H$ by increasing order of $\sigma'$.

    \begin{figure}[h!]
        \centering
        \begin{tikzpicture}[every circle node/.style = {draw}, every edge/.style={draw,sloped}]
            \node[circle] at (-1,0.75) (r) {$r$};
            \node[circle] at (2,3) (w0) {$w_0$};
            \node[circle] at (2,1.5) (w1) {$w_1$};
            \node[circle] at (2,0) (w2) {$w_2$};
            \node[circle] at (2,-1.5) (w3) {$w_{3}$};
            \node at (0.5,-0.7) {$\vdots$};

            \node[circle] at (4,7) (y1) {$y_1$};
            \node[circle] at (4,5.5) (y2) {$y_2$};
            \node[circle] at (4,4) (y3) {$y_3$};
            \node[circle] at (4,1.5) (y4) {$y_4$};
            \node[circle] at (4,0) (y5) {$y_5$};
            \node[circle] at (4,-1.5) (y6) {$y_{6}$};
            \node[circle] at (4,-4) (y7) {$y_7$};
            \node[circle] at (4,-5.5) (y8) {$y_8$};

            \node[circle] at (7,4.5) (u1) {$u_1$};
            \node[circle] at (7,-0.5) (u2) {$u_2$};
            \node[circle] at (7,-5.5) (u3) {$u_3$};

            \foreach \x in {0,1,2,3}
                \draw (r) edge["$m - \x$"] (w\x);
            \draw (u1) edge["$3$"] (u2);
            \draw (u2) edge["$1$"] (u3);
            \draw (u3) edge["$4$"] (y8);
            \draw (u2) edge["$6$"] (y6);
            \draw (u3) edge["$5$"] (y7);
            \draw (u2) edge["$7$"] (y5);
            \draw (u2) edge["$8$"] (y4);
            \draw (u1) edge["$2$",bend left] (u3);
            \draw (u1) edge["$9$"] (y3);
            \draw (u1) edge["$10$"] (y2);
            \draw (u1) edge["$11$"] (y1);
        \end{tikzpicture}
        \caption{Labelling in the case $i = 1$.}
        \label{illus2}
    \end{figure}
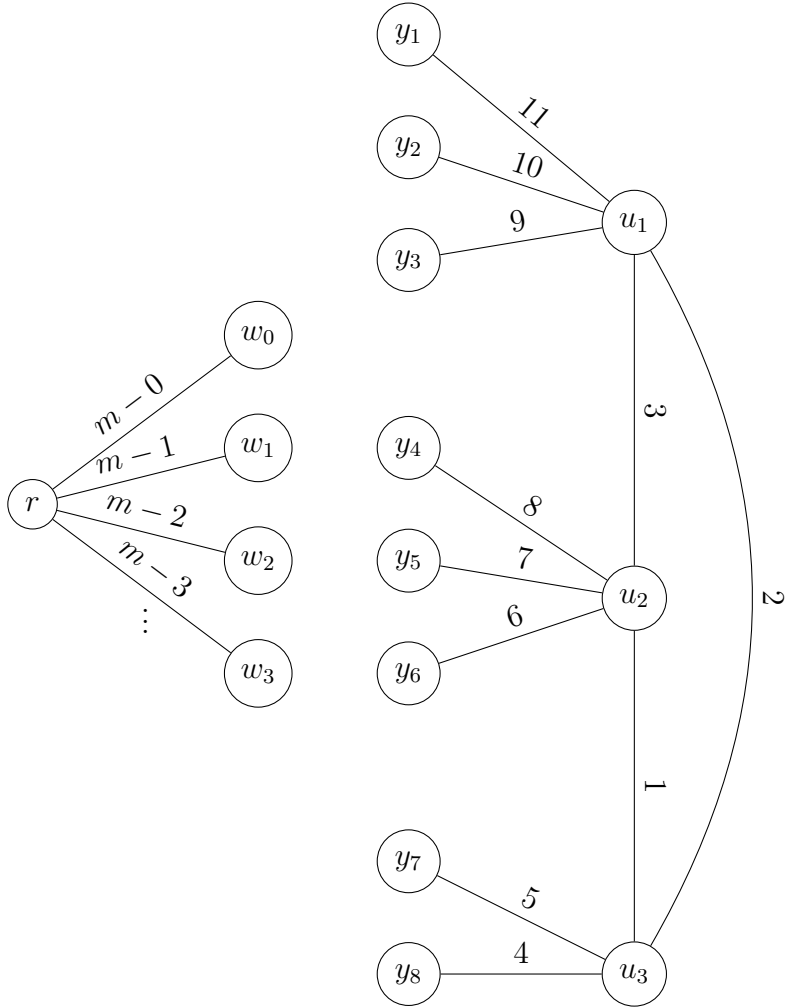

    It is easy (but time and place-consuming, hence the full details of the cases are omitted here) to verify that, whatever structure the graph induced by $\{u_1,u_2,u_3\}$ has, and whatever the values of $d'(u_1), d'(u_2), d'(u_3)$ are, we always obtain $\sigma(u_3) < \sigma(u_2) < \sigma(u_1)$.

    Furthermore, $\sigma(u_1) \leq (2+3) + 10 + 11 +12$. The $(2+3)$ term is an upper bound of the labels of $u_1u_2$ and $u_1u_3$ (if the edges exist), and $12$ is the largest possible label that can be assigned to an edge incident to $u_1$. Overall we obtain $\sigma(u_1) \leq 38$.

    Since the edges incident to $r$ are labelled with the $n-4$ largest labels, we have, for any $v \in H$: $\sigma(r) > \sigma(v) \geq m - (n-5) \geq 6n + 5$ (because $m \geq 7n$), and hence $\sigma(v) \geq 101$ because $n \geq 16$. This means $\sigma(v) > \sigma(u_1)$, and since the vertices of $H$ are sorted by increasing order of $\sigma'$ before labelling the edges incident to $r$, it is impossible to obtain a conflict between two vertices of $H$. Overall, the labelling obtained is antimagic.

    \item $i = 2$.

    We label the graph as illustrated in Figure \ref{illus3}: once again the edges incident to $u_3$ and $u_2$ are labelled with the smallest labels (first the edges incident to $u_3$, then the edges incident to $u_2$). We reserve the labels $\{m-1,m-3,\ldots, m-2(n-5)+1, m-2(n-5) - 1\}$ for the edges incident to $r$, and the labels $\{m,m-2,m-4, \ldots, m-2(d'(u_1)-2), m-2(n-5)-2\}$ for the edges incident to $u_1$. As in all the previous labellings, we sort the vertices in $H$ by increasing order of $\sigma'$ before labelling the edges incident to $r$.

    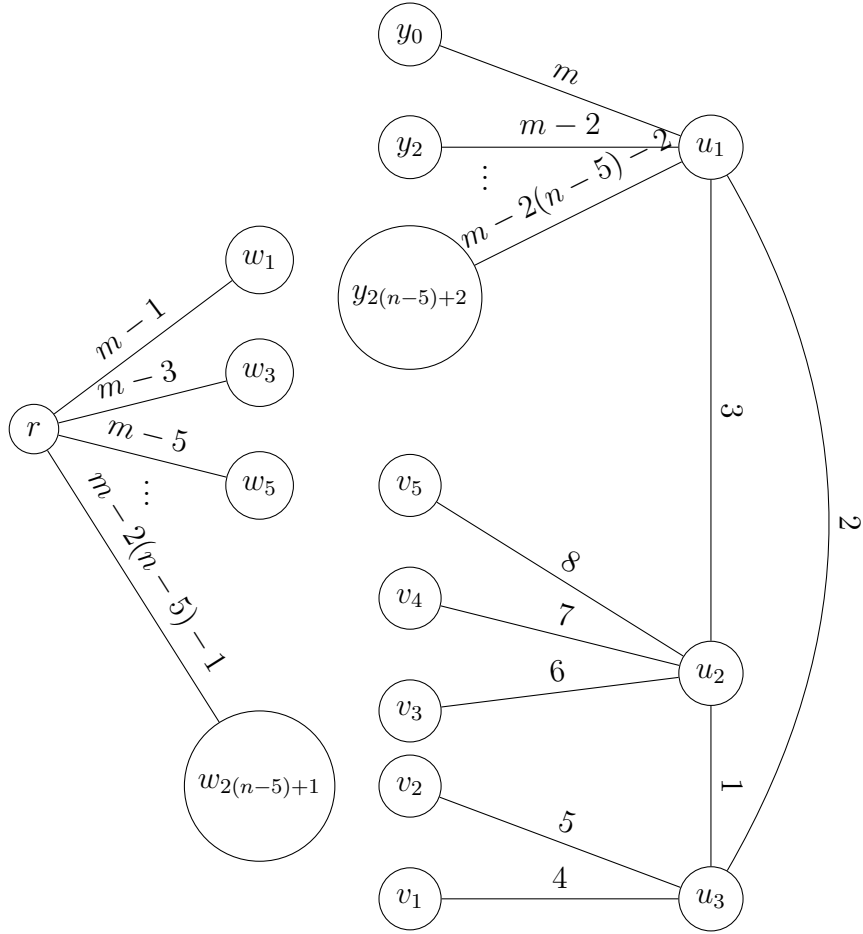
\begin{figure}[h!]
        \centering
        \begin{tikzpicture}[every circle node/.style = {draw}, every edge/.style={draw,sloped}]
            \node[circle] at (-1,0.75) (r) {$r$};
            \node[circle] at (2,3) (w1) {$w_1$};
            \node[circle] at (2,1.5) (w3) {$w_3$};
            \node[circle] at (2,0) (w5) {$w_5$};
            \node at (0.5,0) {$\vdots$};
            \node[circle] at (2,-4) (wn) {$w_{2(n-5) + 1}$};

            \node[circle] at (4,2.5) (yalph) {$y_{2(n-5) + 2}$};
            \node[circle] at (4,0) (v5) {$v_5$};
            \node[circle] at (4,-1.5) (v4) {$v_4$};
            \node[circle] at (4,-3) (v3) {$v_{3}$};
            \node[circle] at (4,-4) (v2) {$v_2$};
            \node[circle] at (4,-5.5) (v1) {$v_1$};

            \node at (5,4.2) {$\vdots$};
            \node[circle] at (4,6) (y0) {$y_0$};
            \node[circle] at (4,4.5) (y2) {$y_2$};

            \node[circle] at (8,4.5) (u1) {$u_1$};
            \node[circle] at (8,-2.5) (u2) {$u_2$};
            \node[circle] at (8,-5.5) (u3) {$u_3$};

            \foreach \x in {1,3,5}
                \draw (r) edge["$m - \x$"] (w\x);
            \draw (u1) edge["$3$"] (u2);
            \draw (u2) edge["$1$"] (u3);
            \draw (u3) edge["$4$"] (v1);
            \draw (u2) edge["$6$"] (v3);
            \draw (u3) edge["$5$"] (v2);
            \draw (u2) edge["$7$"] (v4);
            \draw (u2) edge["$8$"] (v5);
            \draw (u1) edge["$2$",bend left] (u3);
            \draw (u1) edge["$m - 2(n-5) - 2$"] (yalph);
            \draw (r) edge["$m-2(n-5) - 1$"] (wn);
            \draw (u1) edge["$m$"] (y0);
            \draw (u1) edge["$m-2$"] (y2);
        \end{tikzpicture}
        \caption{Labelling in the case $i = 2$.}
        \label{illus3}
    \end{figure}

    Since the labels reserved for the edges incident to $r$ are spaced by 2, we guarantee that, for any two distinct vertices $v_1,v_2 \in H$, $|\sigma(v_1) - \sigma(v_2)| \geq 2$. Moreover, for any $v \in H$, we have $\sigma(v) \geq m - 2(n-5) - 1 \geq 5n + 9 \geq 89$ and $\sigma(u_3) < \sigma(u_2) < 30$. We also have $\sigma(u_1) \geq \sigma(u_2) + 4$.

    Let us now compare the values of $\sigma(r)$ and $\sigma(v)$, for any $v \in H$:

    \begin{align*}
        \sigma(v) \leq & ~ m + (m-1) + (m - 2(d'(u_1) - 1)) + (m - 2d'(u_1)) + \cdots + (m-2(n-5)) \\
        & + (m-2(n-5)-3) + (m-2(n-5)-4) + \cdots + (m-2(n-5)-\alpha)
    \end{align*}

    for some $\alpha \geq 3$. Since $d'(u_1) \geq 4$, we obtain that $\sigma(v) \leq m + (m-1) + (m-6) + (m-8) + \cdots + (m-2(n-5)) + (m-2(n-5)-3)$. Meanwhile, $\sigma(r) = (m-1) + (m-3) + \cdots + (m-2(n-5)-1)$. By comparing the two sums term-to-term, and since $d(r) = n - 4 \geq 12$, we obtain $\sigma(r) \geq \sigma(v) + 4$.

    The only possible conflict in the graph is then between $u_1$ and a vertex $v \in H \cup \{r\}$. In this case, if $v \neq y_0$, we perform the $m \leftrightarrow m - 1$ exchange. This decreases the values of $\sigma(u_1)$ and $\sigma(y_0)$ by $1$, and increases the values of $\sigma(r)$ and $\sigma(w_1)$ by 1. Since any two $v_1,v_2 \in H \cup \{r\}$ are such that $|\sigma(v_1) - \sigma(v_2)| \geq 2$, $u_1$ cannot be in conflict with another vertex $x \neq r,w_1, y_0$ (meaning a vertex $x$ with its value $\sigma(x)$ unchanged after the exchange). If $v = r$, it is impossible that $\sigma(w_1) + 1 = \sigma(u_1) - 1$ since $\sigma(r) \geq \sigma(w_1) + 4$. Otherwise, since $w_1$ is such that $\sigma(w_1) = \max\limits_{y \in H} \sigma(y)$, it is impossible that $\sigma(w_1) + 1 = \sigma(u_1) - 1$ since $u_1$ was in conflict with $v$. Moreover, $v \neq y_0$ and therefore $u_1$ is conflict-free after the exchange is applied. Finally, it is impossible that $\sigma(w_1) + 1 = \sigma(y_0) - 1$ because, again, $w_1$ is the vertex of $H$ with the largest value of $\sigma$. Overall, the labelling obtained is antimagic.

    If $v = y_0$, we perform the $m - 2(n-5) - 1 \leftrightarrow m - 2(n-5) - 2$ exchange instead. This increases $\sigma(u_1)$ and $\sigma(y_{2(n-5)+2})$ by 1, and decreases $\sigma(r)$ and $\sigma(w_{2(n-5) + 1})$ by 1. Since $v = y_0 \neq y_{2(n-5)+2}$, and since $w_{2(n-5) + 1}$ is the vertex of $H$ with the smallest value of $\sigma$, we similarly obtain an antimagic labelling.

    \item $i = 3$.

    We label the graph as illustrated in Figure \ref{illus4}: we use the smallest labels for the edges induced by the graph $\{u_1,u_2,u_3\}$. We then label the edges incident to $u_3$ with the smallest labels available, the edges incident to $u_2$ with $\{m-2,m-5,m-8,\ldots, m-3d'(u_2)-2\}$ and the edges incident to $u_1$ with $\{m-1,m-4,m-7,\ldots,m-3d'(u_1)-1\}$. The edges incident to $r$ are labelled with $\{m,m-3,\ldots,m-3(n-5)\}$. The edges of the graph induced by $V \setminus \{r,u_3\}$ (except the edge $u_1u_2$, if it exists) are labelled with the colouring process described in Section 2. The vertices of $H$ are once again sorted by increasing order of $\sigma'$ before labelling the edges incident to $r$.

    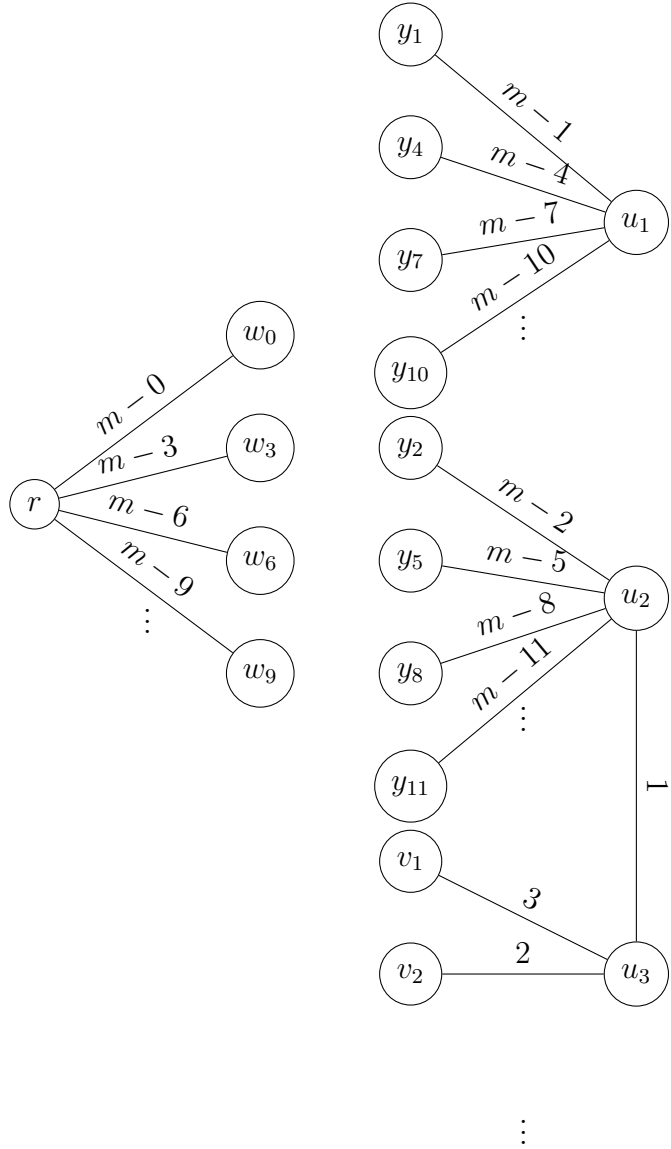
\begin{figure}[h!]
        \centering
        \begin{tikzpicture}[every circle node/.style = {draw}, every edge/.style={draw,sloped}]
            \node[circle] at (-1,0.75) (r) {$r$};
            \node[circle] at (2,3) (w0) {$w_0$};
            \node[circle] at (2,1.5) (w3) {$w_3$};
            \node[circle] at (2,0) (w6) {$w_6$};
            \node[circle] at (2,-1.5) (w9) {$w_{9}$};
            \node at (0.5,-0.7) {$\vdots$};

            \node[circle] at (4,7) (y1) {$y_1$};
            \node[circle] at (4,5.5) (y4) {$y_4$};
            \node[circle] at (4,4) (y7) {$y_7$};
            \node[circle] at (4,2.5) (y10) {$y_{10}$};
            \node[circle] at (4,1.5) (y2) {$y_2$};
            \node[circle] at (4,0) (y5) {$y_5$};
            \node[circle] at (4,-1.5) (y8) {$y_{8}$};
            \node[circle] at (4,-3) (y11) {$y_{11}$};
            \node[circle] at (4,-4) (v1) {$v_1$};
            \node[circle] at (4,-5.5) (v2) {$v_2$};

            \node at (5.5,-7.5) {$\vdots$};
            \node at (5.5,-2) {$\vdots$};
            \node at (5.5,3.2) {$\vdots$};

            \node[circle] at (7,4.5) (u1) {$u_1$};
            \node[circle] at (7,-0.5) (u2) {$u_2$};
            \node[circle] at (7,-5.5) (u3) {$u_3$};

            \foreach \x in {0,3,6,9}
                \draw (r) edge["$m - \x$"] (w\x);
            \foreach \x in {1,4,7,10}
                \draw (u1) edge["$m - \x$"] (y\x);
            \foreach \x in {2,5,8,11}
                \draw (u2) edge["$m - \x$"] (y\x);
            \draw (u3) edge["$3$"] (v1);
            \draw (u3) edge["$2$"] (v2);
            \draw (u2) edge["$1$"] (u3);
        \end{tikzpicture}
        \caption{Labelling in the case $i = 3$.}
        \label{illus4}
    \end{figure}

    Since the labels, for the edges incident to $r$, are spaced by $3$, we obtained, similarly to Property \ref{ecartv}, that for any two distinct vertices $v_1,v_2 \in H$, $|\sigma(v_1) - \sigma(v_2)| \geq 3$.

    We have: $\sigma(r) = m + (m-3) + \cdots + (m-3(n-5)) \geq \sigma(u_1) + 4$ and $\sigma(u_1) \geq \sigma(u_2) + 4$. We also have, for any $v \in H$, $\sigma(v) \leq m + (m-1) + (m-5) + (m-13) + (m-16) + \cdots + (m-3(n-5)-1)$ (because $d'(u_1) \geq d'(u_2) \geq 4$), and, since $n \geq 16$, $\sigma(r) \geq \sigma(v) + 4$. Finally, $\sigma(u_3) \leq 3 + 4 + 5 + 6 \leq 18$ hence $\sigma(u_3) \leq \sigma(u_2) - 4$ and, for any $v \in H, \sigma(u_3) \leq \sigma(v) - 4$.

    These properties imply that the only possible conflicts are between $u_1$ and some vertex $v_1 \in H$, and between $u_2$ and some vertex $v_2 \in H$. Moreover, this remains true after any of the exchanges that we will perform, since the value of $\sigma(x)$ for any $x \in V$ will  be modified by $0$, $1$ or $-1$. We show the possible exchanges of labels in Table \ref{tab2}, and we distinguish between three cases:

    \begin{table}[h!]
        \centering
        \begin{tabular}{|c|c|c|c|}
            \hline
             Notation & Exchanges & Evolution of $\sigma(u_i)$ & Evolution of $\sigma(v_i)$ \\
             \hline
             $\lambda_{1}$ & $m-1 \leftrightarrow m-2$ & \multirow{4}{*}{$\sigma(u_1) - 1, \sigma(u_2) + 1$} & $\sigma(y_1) - 1, \sigma(y_2) + 1$\\
             $\lambda_{4}$ & $m-4 \leftrightarrow m-5$ & & $\sigma(y_4) - 1, \sigma(y_5) + 1$ \\
             $\lambda_{7}$ & $m-7\leftrightarrow m-8$ & & $\sigma(y_7) - 1, \sigma(y_{8}) + 1$ \\
             $ \lambda_{10}$ & $m - 10 \leftrightarrow m - 11$ & & $\sigma(y_{10}) - 1$, $\sigma(y_{11}) + 1$ \\
             \hline
             $\mu_{0}$ & $m \leftrightarrow m-1$ & \multirow{4}{*}{$\sigma(u_1) + 1$} & $\sigma(y_1) + 1, \sigma(w_0) - 1$ \\
             $ \mu_{3}$ & $m-3 \leftrightarrow m-4$ & & $\sigma(y_4) + 1, \sigma(w_3) - 1$ \\
             $ \mu_{6}$ & $m-6\leftrightarrow m-7$ & & $\sigma(y_7) + 1, \sigma(w_6) - 1$ \\
             $ \mu_{9}$ & $m-9 \leftrightarrow m - 10$ & & $\sigma(y_{10}) + 1, \sigma(w_{9}) - 1$  \\
             \hline
             $\rho_{2}$ & $m-2 \leftrightarrow m-3$ & \multirow{4}{*}{$\sigma(u_2) - 1$} & $\sigma(y_2) - 1, \sigma(w_3) + 1$ \\
             $\rho_{5}$ & $m-5 \leftrightarrow m-6$ & & $\sigma(y_5) - 1, \sigma(w_6) + 1$ \\
             $\rho_{8}$ & $m-8 \leftrightarrow m-9$ & & $\sigma(y_{8}) - 1, \sigma(w_{9}) + 1$ \\
             $ \rho_{11}$ & $m - 15 \leftrightarrow m - 12$ & & $\sigma(y_{11}) - 1, \sigma(w_{12}) + 1$ \\
             \hline
        \end{tabular}
        \caption{Exchanges in the case $i = 3$ (changes made to the value of $\sigma(r)$ are ignored).}
        \label{tab2}
    \end{table}

    \begin{itemize}
        \item \cf{u_1}{v_1}, \cf{u_2}{v_2}.

        Let us consider the $\lambda_j$ exchange, for some $j$: this yields an antimagic labelling unless $v_1 = y_j$ or $v_2 = y_{j+1}$. There are two possible issues and four possible exchanges (each issue can be associated with at most one $\lambda_j$ exchange), it is then possible to obtain an antimagic labelling.

        \item \cf{u_1}{v_1}

        Let us consider the $\mu_j$ exchange, for some $j$: this yields an antimagic labelling unless $v_1 = y_{j+1}, v_2 = y_{j+1}$ (and $\sigma(u_2) = \sigma(y_{j+1}) + 1$) or $v_2 = w_j$ (and $\sigma(u_2) = \sigma(w_j) - 1$). These last two equalities are mutually exclusive (since the values of the $\sigma(v_j)$ are spaced by at least $3$), hence there are two possible issues and four possible exchanges (each issue can be associated with at most one $\mu_j$ exchange), and we can obtain an antimagic labelling.

        \item \cf{u_2}{v_2}

        Let us consider the $\rho_j$ exchange: this yields an antimagic labelling unless $v_2 = y_j$, $v_1 = y_j$ (and $\sigma(u_1) = \sigma(y_j) - 1$), or $v_1 = w_{j+1}$ (and $\sigma(u_1) = \sigma(w_{j+1}) + 1$). These last two equalities are mutually exclusive, hence there are two possible issues and four possible exchanges (each issue can be associated with at most one $\rho_j$ exchange), and we can obtain an antimagic labelling.
        
    \end{itemize}
    
\end{itemize}

\section{Non-connected case}

In Sections 2,3, and 4, we only considered connected graphs. First notice that, if a given graph $G$ has two isolated vertices, or an isolated edge, it cannot be antimagic (we assume that if a vertex $x \in V$ is isolated, $\sigma(x) = 0$). There are two other families of non-connected graphs $G$, such that $\Delta(G) = n - 4$ and $m \geq 7n$: 

\begin{itemize}
    \item $u_3$ is an isolated vertex, and $d'(u_1) > 0$. In this case, the construction described in Section 4 is unchanged and we can obtain an antimagic labelling.
    
    \item $\{u_1,u_2,u_3\}$ induces a connected component, either a $K_3$ or a $P_3$ ($K_3$ being the clique on $3$ vertices, and $P_3$ the induced path on $3$ vertices), hence $d'(u_1) = d'(u_2) = d'(u_3) = 0$. Then $H \cup \{r\}$ is the second connected component, where $r$ is universal. Since $m \geq 7n$, $m_H \geq 6n \geq 2$. We can label the component induced by $\{u_1,u_2,u_3\}$ with the smallest labels and apply the construction described in the proof of Theorem \ref{n-1} to the other component, and so we obtain an antimagic labelling.
\end{itemize}

Hence we obtain the following result:

\begin{corollary}
    Let $G$ be a graph that contains no isolated edge and at most one isolated vertex. If $\Delta(G) = n - 4$ and $m \geq 7n$, then $G$ is antimagic.
\end{corollary}

\section{Conclusion}

We have proven that, for any graph $G$ such that $m \geq 7n$ and $\Delta(G) = n - 4$, $G$ is antimagic. The proof detailed in Sections 2,3,4 shows a new framework to prove the antimagicness of some graphs, applied here to graphs with large maximum degree and small average degree. It is worth noting the framework used here, with a post-processing of the labelling using exchange between labels, is somewhat similar to the one we used in \cite{beaudoire2025} to show the antimagicness of graphs with a dominating clique. It is likely that the same ideas could be used to prove the antimagicness of other classes of graphs.

For graphs with large average degree, Eccles proved the following in \cite{eccles2016}:

\begin{theorem}
    \label{ecclesed}
    There exists an absolute constant $d_0$ such that, if G is a graph with average degree at least $d_0$, and $G$ contains no isolated edge and at most one isolated vertex, then $G$ is antimagic.
\end{theorem}

Note that the proof of our construction shows a similar result over graphs with average degree at least $14$ (since we assumed $m \geq 7n$).

It is likely one can refine the result of Eccles, which applies to graphs with a very large number of vertices and a very large average degree (note that the proof described in \cite{eccles2016} induces an upper bound on the average degree $d_0$ of $4182$), to prove a result akin to the following conjecture:

\begin{conj}
    Let $k$ be a positive integer, and let $G$ be a connected graph with $\Delta(G) = n - k$. If $m \geq f(k)n$, for some positive function $f$, then $G$ is antimagic.
\end{conj}

Notice that the cases $k = 1$ and $k = 2$ are known to be true, with $f(k) = 0$ in both cases; the case $k = 3$ is also known to be true (see Theorem \ref{yilma}) with $f(k) = 0$ and $n \geq 9$. Moreover, the case $k = 4$ is proved to be true with Theorem \ref{n-4}, with $f(4) \leq 7$. Finally, Theorem \ref{ecclesed} shows that $f(k) \leq \frac{d_0}{2}$ for graphs $G$ with $n -k \geq d_0$.

Note that our construction does not require any condition over the number of vertices in the graph (contrary to Conjecture \ref{conjalon}); however, the $m \geq f(k)n$ hypothesis ($m \geq 7n$ in Theorem \ref{n-4}) does induce a (constant) lower bound over $n$ ($n \geq 16$ in Theorem \ref{n-4}).

A question that arises naturally would be to study if it is possible to show the antimagicness of 'sparse' graphs, for some definition of sparse, and ideally for graphs with $m < 7n$, thus proving the antimagicness of all graphs $G$ with $\Delta(G) = n - 4$. Another idea would be to build a framework such that the values of $\sigma(u_1)$, $\sigma(u_2)$ and $\sigma(u_3)$, meaning the labels assigned to their incident edges, depend on their respective degree to try and 'guess' where they would naturally fall among the other values of $\sigma(v), v \in H$, similarly to the work of Yilma in \cite{yilma2013}.

\bibliographystyle{elsarticle-num}
\bibliography{biblio}

\end{document}